\documentclass[11pt]{article}

\usepackage{amsmath}
\usepackage{latexsym}
\usepackage{amssymb}
\usepackage{graphicx}
\numberwithin{equation}{section}

\newtheorem{theorem}{\bf Theorem}[section]
\newtheorem{definition}{\bf Definition}[section]
\newtheorem{lemma}{\bf Lemma}[section]

\def\udots{\mathinner{\mkern1mu\raise-1pt\vbox{\kern7pt\hbox{.}}\mkern2mu
    \raise2pt\hbox{.}\mkern2mu\raise5pt\hbox{.}\mkern1mu}}

\allowdisplaybreaks

\begin{document}
\begin{center}
{\Large \bf Absolute Continuity of the Laws of Perturbed Diffusion Processes and Perturbed Reflected Diffusion Processes}
\end{center}
\begin{center}
Wen Yue, Tusheng Zhang
\end{center}
\begin{center}
{\scriptsize Department of Mathematics, University of Manchester, Oxford Road, Manchester M13 9PL, England, UK}
\end{center}
\begin{abstract}
In this paper, we prove that the laws of perturbed diffusion processes and perturbed reflected diffusion processes are absolutely continuous with respect to the Lebesgue measure. The main tool we use is the Malliavin calculus.t
\end{abstract}


\section{Introduction}
       Let $(\Omega, \mathcal{F}, \{\mathcal{F}_{t}\}_{t\geq 0},P)$ be a filtered probability space with filtration$\{\mathcal{F}_{t}\}_{t\geq 0}$ satisfying the usual conditions. $\{B_{t}\}_{t\geq0}$ is a one-dimensional standard $\{\mathcal{F}_{t}\}_{t\geq 0}$-Brownian Motion. Suppose that $\sigma(x),b(x)$ are Lipschitz continuous functions on R. There now exists a considerable body of literature devoted to the study of perturbed stochastic differential equations(SDEs), see e.g. [1-4], [6-8], [11], [13]. It was proved in \cite{RDZ} that the following perturbed SDE:
       \begin{equation}
       Y_{t}=y_{0}+\int_{0}^{t}\sigma(Y_{s})dB_{s}+\int_{0}^{t}b(Y_{s})ds+\alpha \max_{0\leq s\leq t} Y_{s}. \label{perturbed diffusion}
       \end{equation}
       and the perturbed reflected SDE:
       \begin{equation}
      \left\{ \begin{array}{lll}
      X_{t}=\int_{0}^{t}\sigma(X_{s})dB_{s}+\int_{0}^{t}b(X_{s})ds+\alpha \max\limits_{0\leq s \leq t}X_{s}+L_{t}\\
      X_t\geq 0\\
      \int_0^t\chi_{\{X_s=0\}}dL_s=L_t.\end{array}\right. \label{reflected PD}
      \end{equation}
      admit unique solutions. Perturbed Brownian motion arose in a study of the windings of planar Brownian motion, see \cite{GY1}. Perturbed diffusion processes are also continuous versions of self-interacting random walks.
      \vskip 0.3cm
       The purpose of this paper is to establish the absolute continuity of the laws of perturbed diffusion processes as well as perturbed reflected diffusion processes under appropriate  conditions. The absolute continuity of the laws of the solutions is of fundamental importance both theoretically and numerically.
      The absolute continuity of the laws of the solutions to stochastic differential equations has been studied by many people. We refer the reader to the books \cite{N}, \cite{S} and references therein.
       \vskip 0.3cm
       The tool we use is naturally Malliavin Calculus. Because the extra terms in equation (\ref{perturbed diffusion}) and (\ref{reflected PD}) involve the maximum of the solution itself, the Malliavin differentiability of the solutions becomes very delicate. For the absolute continuity of the laws of the solutions, we need a careful analysis of the time points where the solution $X$ reaches its maximum. The local property of the Malliavin derivative and a comparison theorem for stochastic differential equations play a crucial role.

       \vskip 0.3cm
       This paper is organized as follows. In Section 2, we collect some results of Malliavin  calculus to be used later in the paper. In Section 3, we prove that the perturbed diffusion process is Malliavin differentiable
       and establish the absolute continuity of the laws of the perturbed diffusion processes.  In Section 4, we study the reflected perturbed diffusion processes. The Malliavin differentiability and the absolute continuity of the solutions are obtained.

\section{Preliminaries}
In this section, we collect some results on Malliavin calculus which will be used in the paper.\\
Let $\Omega=C_{0}(R_{+})$ be the space of continuous functions on $R_{+}$ which are zero at zero. Denote by $\mathcal{F}$ the Borel $\sigma$-field on $\Omega$ and P the Wiener measure. Then the canonical coordinate process \{$B_{t}$, $t\in R_{+}$\} on $\Omega$ is a Brownian motion. Define $\mathcal{F}_{t}^{0}=\sigma(B_{s},s\leq t).$ Denote by $\mathcal{F}_{t}$ the completion of $\mathcal{F}_{t}^{0}$ with respect to the $P$-null sets of $\mathcal{F}.$ Let $h\in L^{2}(R_{+}).$ $W(h)$ will stand for the Wiener integral as follows:
\begin{equation}
W(h)=\int_{0}^{\infty}h(t)dB_{t}.
\end{equation}
$\{W(h),h\in H\}$ is a Gaussian Process on $H:=L^{2}(R_{+},\mathcal{B},\mu)$, where $(R_{+},\mathcal{B})$ is a measurable space, $\mathcal{B}$ is the Borel sigma field of $R_{+}$ and $\mu$ is the Lebesgue measure on $R_{+}$.

We denote by $C_{p}^{\infty}(\mathbb{R}^{n})$ the set of all infinitely continuously differentiable functions $f:\mathbb{R}^{n}\rightarrow\mathbb{R}$ such that $f$ and all of its partial derivatives have polynomial growth.
Let $S$ be the set of smooth random variables defined by
\begin{equation}
S=\{F=f(W(h_{1}),W({h_{2}}),...,W(h_{n}));\ h_{1},...,h_{n}\in L^{2}(R_{+}),n\geq 1, f\in C_{p}^{\infty}(R^{n})\}.
\end{equation}
Let $F\in S$. Define its Malliavin derivative $D_{t}F$ by
\begin{equation}
D_{t}F=\Sigma_{i=1}^{n}\partial_{i}f(W(h_{1}),W(h_{2}),...,W(h_{n}))h_{i}(t),
\end{equation}
and its norm by
\begin{equation}
||F||_{1,2}=[E(|F|^{2})+E(||DF||_{H}^{2})]^{\frac{1}{2}}.
\end{equation}
Let $\mathbb{D}^{1,2}$ be the completion of $S$ under the norm $||.||_{1,2}.$
The following result is from \cite{N}.
\begin{theorem}
Let $F\in \mathbb{D}^{1,2}.$ If $||DF||_{H}>0$ a.s., then the law of the random variable F is absolutely continuous with respect to Lebesgue measure.
\end{theorem}

\section{Absolute continuity of the laws of perturbed diffusion processes}

Let $\sigma(x),b(x)$ be Lipschitz continuous functions on R, i.e., there exists a constant C such that
 \begin{equation}
  |\sigma(x)-\sigma(y)| \leq C|x-y|,
  \end{equation}
  \begin{equation}
   |b(x)-b(y)| \leq C|x-y|.
   \end{equation}
   For $\alpha < 1$, $y_{0}\in R$, consider the following stochastic differential equation:
    \begin{equation}
   Y_{t}=y_{0}+\int_{0}^{t}\sigma(Y_{s})dB_{s}+\int_{0}^{t}b(Y_{s})ds+\alpha \max_{0\leq s\leq t} Y_{s}. \label{eq1}
\end{equation}
It was shown in \cite{RDZ} that equation (\ref{eq1}) admits a unique, continuous, adapted solution. We have the following result.
\begin{theorem}
 Let $Y_{t}$ be the unique solution to equation (\ref{eq1}). Then $Y_{t} \in \mathbb{D}^{1,2}$ for any $t\geq 0$.
\end{theorem}
\begin{proof}
 Consider Picard approximations given by
  \begin{equation}
   Y_{t}^{0}=\frac{y_{0}}{1-\alpha}, \  0\leq t<\infty.
    \end{equation}
    For $n\geq0$, define $Y_{t}^{n+1}$ to be the unique, continuous, adapted solution to the following equation:\\
   \begin{equation}
    Y_{t}^{n+1}=y_{0}+\int_{0}^{t}\sigma(Y_{s}^{n})dB_{s}+\int_{0}^{t}b(Y_{s}^{n})ds+\alpha \max_{0\leq s\leq t}Y_{s}^{n+1}.
     \end{equation}
Such a solution exists and can be expressed explicitly as
 \begin{equation} Y_{t}^{n+1}=\frac{y_{0}}{1-\alpha}+\int_{0}^{t}\sigma(Y_{s}^{n})dB_{s}+\int_{0}^{t}b(Y_{s}^{n})ds+\frac{\alpha}{1-\alpha} \max_{0\leq s\leq t}(\int_{0}^{s}\sigma(Y_{u}^{n})dB_{u}+\int_{0}^{s}b(Y_{u}^{n})du) \label{Y(n+1)}
  \end{equation}
It was shown in \cite{RDZ} that the solution $Y_{t}$ is the limit of $Y_{t}^{n}$ in $L^{2}(\Omega).$ \\
We will prove the following property by induction on n:\\
(\textbf{P}) \ \  $Y_{t}^{n}\in \mathbb{D}^{1,2}, E[\int_{0}^{t}||DY_{u}^{n}||_{H}^{2}du]< \infty, t \geq 0.$\\
Clearly, (\textbf{P}) holds for n=0. Suppose $Y_{t}^{n}\in\mathbb{D}^{1,2}$ and $E[\int_{0}^{t}||DY_{u}^{n}||_{H}^{2}du]< \infty $. Applying
Proposition 1.2.4 in \cite{N} to the random variable $Y_{s}^{n}$ and to $\sigma$ and $b$, we deduce that the random variables $\sigma(Y_{s}^{n})$ and $b(Y_{s}^{n})$ belong to $\mathbb{D}^{1,2}$ and that there exist adapted processes $\overline{\sigma}^{n}(s)$ and $\overline{b}^{n}(s)$, which are uniformly bounded by some constant K, such that:\\
\begin{equation} \label{eq1:eps}
D_{r}(\sigma(Y_{s}^{n}))=\overline{\sigma}^{n}(s)D_{r}(Y_{s}^{n})I_{\{r\leq s\}},
\end{equation}
and
\begin{equation}\label{eq2:eps}
D_{r}(b(Y_{s}^{n}))=\overline{b}^{n}(s)D_{r}(Y_{s}^{n})I_{\{r\leq s\}}.
\end{equation}
From (\ref{eq1:eps}) and (\ref{eq2:eps}) we get
\begin{equation}
|D_{r}(\sigma(Y_{s}^{n}))|\leq K |D_{r}(Y_{s}^{n})|,
\end{equation}
and
\begin{equation}
|D_{r}(b(Y_{s}^{n}))|\leq K |D_{r}(Y_{s}^{n})|.
\end{equation}
By Lemma 1.3.4 in \cite{N}, we conclude that
\begin{equation}
\int_{0}^{t}\sigma(Y_{s}^{n})dB_{s} \in \mathbb{D}^{1,2}. \label{the second element}
\end{equation}
For $r\leq t$, by Proposition 1.3.8 in \cite{N}, \\
\begin{equation}
D_{r}[\int_{0}^{t}\sigma(Y_{s}^{n})dB_{s}]=\sigma(Y_{r}^{n})+\int_{r}^{t}D_{r}(\sigma(Y_{s}^{n}))dB_{s}
\end{equation}
Similarly, we have
\begin{equation}
\int_{0}^{t}b(Y_{s}^{n})ds \in \mathbb{D}^{1,2}, \label{the third element}
\end{equation}
\begin{equation}
D_{r}[\int_{0}^{t}b(Y_{s}^{n})ds]= \int_{r}^{t}D_{r}(b(Y_{s}^{n}))ds.
\end{equation}
Let $Z_{s}^{n}=\int_{0}^{s}\sigma(Y_{u}^{n})dB_{u}$, $X_{s}^{n}=\int_{0}^{s}b(Y_{u}^{n})du$.
Then
\begin{equation}
Z_{s}^{n}+X_{s}^{n}\in \mathbb{D}^{1,2},
\label{eq2}
\end{equation}
and
\begin{equation}
E[\sup_{0\leq s\leq t}(Z_{s}^{n}+X_{s}^{n})^{2}]\leq E[\sup_{0\leq s\leq t}2((Z_{s}^{n})^{2}+(X_{s}^{n})^{2})]\leq 2E[\sup_{0\leq s\leq t}(Z_{s}^{n})^{2}]+2E[\sup_{0\leq s\leq t}(X_{s}^{n})^{2}]< \infty.
\label{eq3}
\end{equation}
Next we show that
\begin{equation}
 E[\sup_{0\leq s\leq t}||D(Z_{s}^{n}+X_{s}^{n})||_{H}^{2}]< \infty.
\label{eq4}
 \end{equation}
Now
\begin{eqnarray}
\nonumber
 E[\sup_{0\leq s\leq t}||D(Z_{s}^{n}+X_{s}^{n})||_{H}^{2}]
&=& E[\sup_{0\leq s\leq t}\int_{0}^{s}|D_{r}(Z_{s}^{n}+X_{s}^{n})|^{2}dr]\\ \nonumber
&\leq& 3E\{\sup_{0\leq s\leq t}\int_{0}^{s}[\sigma(Y_{r}^{n})^{2}+|\int_{r}^{s}D_{r}(\sigma(Y_{u}^{n}))dB_{u}|^{2}\\ \nonumber
&&+|\int_{r}^{s}D_{r}(b(Y_{u}^{n}))du|^{2}]dr\}\\ \nonumber
&\leq& 3E\int_{0}^{t}\sigma(Y_{r}^{n})^{2}dr+3\int_{0}^{t}E[\sup_{r\leq s \leq t}|\int_{r}^{s}D_{r}(\sigma(Y_{u}^{n}))dB_{u}|^{2}]dr\\ \nonumber
&&+3E\int_{0}^{t}[\int_{r}^{t}|D_{r}(b(Y_{u}^{n}))|du]^{2}dr\\ \nonumber
&\leq& 3\int_{0}^{t}E[\sigma(Y_{r}^{n})^{2}]dr+3C\int_{0}^{t} \int_{r}^{t}E[D_{r}(\sigma(Y_{u}^{n}))]^{2}dudr \\ \nonumber
&&+3\int_{0}^{t}\int_{r}^{t}E[D_{r}(b(Y_{u}^{n}))^{2}]du(t-r)dr \\ \nonumber
&\leq& 3\int_{0}^{t}E[\sigma(Y_{r}^{n})^{2}]dr+3CK^{2}\int_{0}^{t}\int_{r}^{t}E[D_{r}(Y_{u}^{n})^{2}]dudr\\ \nonumber
&&+3K^{2}\int_{0}^{t}\int_{r}^{t}E[D_{r}(Y_{u}^{n})^{2}]du(t-r)dr \\ \nonumber
&\leq&3\int_{0}^{t}E[\sigma(Y_{r}^{n})^{2}]dr
+(3CK^{2}+3K^{2}t)\int_{0}^{t}\int_{r}^{t}E[D_{r}(Y_{u}^{n})^{2}]dudr \\
&<& \infty.   \label{ineq1}
\end{eqnarray}
So we have proved (\ref{eq4}).\\
From (\ref{eq2}),(\ref{eq3}) and (\ref{eq4}), and by Proposition 2.1.10 in \cite{N}, we conclude
\begin{equation}
\max_{0\leq s\leq t}(Z_{s}^{n}+X_{s}^{n})\in  \mathbb{D}^{1,2}, \label{the fourth element}
\end{equation}
and
\begin{equation}
E[||D(\max_{0\leq s\leq t}(Z_{s}^{n}+X_{s}^{n}))||_{H}^{2}]\leq E[\max_{0\leq s\leq t}||D(Z_{s}^{n}+X_{s}^{n})||_{H}^{2}]. \label{ineq2}
\end{equation}
It follows from (\ref{Y(n+1)}) that $Y_{t}^{n+1}\in \mathbb{D}^{1,2}$. Moreover,
\begin{eqnarray}
E\int_{0}^{t}||D(Y_{u}^{n+1})||_{H}^{2}du \nonumber
    &\leq& 4\int_{0}^{t}\int_{0}^{u}E[\sigma(Y_{r}^{n})^{2}]drdu
         +4\int_{0}^{t}\int_{0}^{u}E(\int_{r}^{u}D_{r}(\sigma(Y_{v}^{n}))dB_{v})^{2}drdu \\ \nonumber
&&   +4\int_{0}^{t}\int_{0}^{u}E[\int_{r}^{u}D_{r}(b(Y_{v}^{n}))dv]^{2}drdu\\ \nonumber
&&       +4(\frac{\alpha}{1-\alpha})^{2}\int_{0}^{t}\int_{0}^{u}E[D_{r}(\sup_{0\leq v\leq u}(Z_{v}^{n}+X_{v}^{n}))]^{2}drdu \\ \nonumber
&\leq& 4t\int_{0}^{t}E[\sigma(Y_{r}^{n})^{2}]dr
       +4K^{2}(t+1)\int_{0}^{t}\int_{0}^{u}\int_{r}^{u}E(D_{r}(Y_{v}^{n}))^{2}dvdrdu\\ \nonumber
&&    +4(\frac{\alpha}{1-\alpha})^{2}\int_{0}^{t}\int_{0}^{u}E[D_{r}(\sup_{0\leq v\leq u}(Z_{v}^{n}+X_{v}^{n}))]^{2}drdu \\ \nonumber
&\leq& 4t\int_{0}^{t}E[\sigma(Y_{r}^{n})^{2}]dr+4K^{2}(t+1)\int_{0}^{t}[E\int_{0}^{u}||D(Y_{v}^{n})||_{H}^{2}dv]du \\ \nonumber
&& +4(\frac{\alpha}{1-\alpha})^{2}\int_{0}^{t}E||D(\sup_{0\leq v \leq u}(Z_{v}^{n}+X_{v}^{n}))||_{H}^{2}du \\ \nonumber
&\leq& 4t\int_{0}^{t}E[\sigma(Y_{r}^{n})^{2}]dr+4K^{2}(t+1)\int_{0}^{t}[E\int_{0}^{u}||D(Y_{v}^{n})||_{H}^{2}dv]du \\ \nonumber
&& +4(\frac{\alpha}{1-\alpha})^{2}\int_{0}^{t}E[\sup_{0\leq v \leq u}||D(Z_{v}^{n}+X_{v}^{n})||_{H}^{2}]du \\
&<& \infty, \label{1}
\end{eqnarray}
where (\ref{eq4}) has been used. Property (\textbf{P}) is proved.\\
Now we prove
 \begin{equation}
 \sup_{n}E(||DY_{t}^{n}||_{H}^{2})< \infty.
 \end{equation}
Note that
\begin{eqnarray}
D_{r}(Y_{t}^{n})
&=&\sigma(Y_{r}^{n-1})+\int_{r}^{t}D_{r}(\sigma(Y_{s}^{n-1}))dB_{s}+\int_{r}^{t}D_{r}(b(Y_{s}^{n-1}))ds\nonumber\\ \nonumber
&&+\frac{\alpha}{1-\alpha}D_{r}[\max_{0\leq s \leq t}(Z_{s}^{n-1}+X_{s}^{n-1})].
\end{eqnarray}
 \begin{eqnarray}
E(||DY_{t}^{n}||_{H}^{2}) \nonumber
&=& E\int_{0}^{t}|D_{r}Y_{t}^{n}|^{2}dr \\ \nonumber
&\leq&4\{E[\int_{0}^{t}|\sigma(Y_{r}^{n-1})|^{2}dr]+E[\int_{0}^{t}|\int_{r}^{t}D_{r}(\sigma(Y_{s}^{n-1}))dB_{s}|^{2}dr]\\ \nonumber
&&+E[\int_{0}^{t}|\int_{r}^{t}D_{r}(b(Y_{s}^{n-1}))ds|^{2}dr]\\ \nonumber
&&+(\frac{\alpha}{1-\alpha})^{2}E\int_{0}^{t}|D_{r}\max_{0\leq s \leq t}[\int_{0}^{s}\sigma(Y_{u}^{n-1})dB_{u}+\int_{0}^{s}b(Y_{u}^{n-1})du]|^{2}dr\} \\ \nonumber
&\leq&4\int_{0}^{t}E|\sigma(Y_{r}^{n-1})|^{2}dr+4\int_{0}^{t}E(\int_{r}^{t}|D_{r}(\sigma(Y_{s}^{n-1}))|^{2}ds)dr\\ \nonumber
&&+4t\int_{0}^{t}E[\int_{r}^{t}|D_{r}(b(Y_{s}^{n-1})|^{2}ds]dr \\ \nonumber
&&+4(\frac{\alpha}{1-\alpha})^{2}E[\sup_{0\leq s \leq t}||D(\int_{0}^{s}\sigma(Y_{u}^{n-1})dB_{u}+\int_{0}^{s}b(Y_{u}^{n-1})du)||_{H}^{2}]\\ \nonumber
&\leq &4E[\int_{0}^{t}|\sigma(Y_{r}^{n-1})|^{2}dr]+4K^{2}\int_{0}^{t}E[\int_{r}^{t}|D_{r}(Y_{s}^{n-1})|^{2}ds]dr\\ \nonumber
 &&+4K^{2}t\int_{0}^{t}E[\int_{r}^{t}|D_{r}(Y_{s}^{n-1})|^{2}ds]dr\\ \nonumber
 &&+4(\frac{\alpha}{1-\alpha})^{2}\{3\int_{0}^{t}E(\sigma(Y_{r}^{n-1}))^{2}dr\\ \nonumber
 &&+(3CK^{2}+3K^{2}t)\int_{0}^{t}\int_{r}^{t}E(D_{r}(Y_{u}^{n-1}))^{2}dudr\}\\ \label{ineq3}
 &\leq & C_{1}\int_{0}^{t}E|\sigma(Y_{r}^{n-1})|^{2}dr+C_{2}\int_{0}^{t}E||DY_{u}^{n-1}||_{H}^{2}du.
 \end{eqnarray}
Where (\ref{eq4}) and (\ref{ineq2}) were used.\\
Note that
$A=\sup\limits_{n}\int_{0}^{t}E|\sigma(Y_{r}^{n-1})|^{2}dr\leq C\sup\limits_{n}\int_{0}^{t}E(1+|Y_{r}^{n-1}|^{2})dr<\infty$, because $Y_{n}$ converges to $Y$ uniformly w.r.t time parameter from \cite{RDZ}.\\
Iterating (\ref{ineq3}) gives $\sup\limits_{n}E||DY_{t}^{n}||_{H}^{2}<\infty.$ \\
Thus by Lemma 1.2.3 in \cite{N} we deduce that $Y_{t}\in \mathbb{D}^{1,2}$ and $DY_{t}^{n}\rightarrow DY_{t}$ weakly in $L^{2}(\Omega;H)$. $\Box$
\end{proof}

\begin{theorem}
Assume that $\sigma(\cdot)$ and $b(\cdot)$ are Lipschitz continuous, and $|\sigma(x)|>0$, for all $x \in R$. Then, for $t>0$, the law of $Y_{t}$ is absolutely continuous with respect to Lebesgue measure.
\end{theorem}

\begin{proof}
According to Theorem 2.1, we just need to show that
$||DY_{t}||_{H}^{2}>0$ a.s..\\
Note that,
\begin{eqnarray}
D_{r}Y_{t}
&=&\sigma(Y_{r})+\int_{r}^{t}D_{r}(\sigma(Y_{s}))dB_{s}+\int_{r}^{t}D_{r}(b(Y_{s}))ds+\alpha D_{r}(\max_{0\leq s\leq t}Y_{s}),\  r\leq t \label{DYt}
\end{eqnarray}
We have,
\begin{eqnarray}
(D_{r}Y_{t})^{2} \nonumber
&\geq & \frac{1}{2}\sigma(Y_{r})^{2}-[\int_{r}^{t}D_{r}(\sigma(Y_{s}))dB_{s}+\int_{r}^{t}D_{r}(b(Y_{s}))ds+\alpha D_{r}(\max_{0\leq s\leq t}Y_{s})]^{2}\\ \nonumber
&\geq&\frac{1}{2}\sigma(Y_{r})^{2}
-3\{[\int_{r}^{t}D_{r}(\sigma(Y_{s}))dB_{s}]^{2}+[\int_{r}^{t}D_{r}(b(Y_{s}))ds]^{2}+\alpha^{2}[D_{r}(\max_{0\leq s \leq t}Y_{s})]^{2}\}. \\ \nonumber
\end{eqnarray}
Since $\sigma(Y_{r})^{2}$ is continuous w.r.t r,
 it follows that
 \begin{eqnarray}
 \lim_{\epsilon \rightarrow 0}\frac{1}{\epsilon}\int_{t-\epsilon}^{t}\sigma(Y_{r})^{2}dr=\sigma(Y_{t})^{2}. \label{con1}
\end{eqnarray}
Now,
\begin{eqnarray} \nonumber
&&E\{\int_{t-\epsilon}^{t}[\int_{r}^{t}D_{r}(b(Y_{s}))ds]^{2}dr\}\\ \nonumber
&\leq &  \int_{t-\epsilon}^{t}E[\int_{r}^{t}|D_{r}(b(Y_{s}))|^{2}ds (t-r)]dr\\ \nonumber
&\leq& K^{2}\int_{t-\epsilon}^{t}\int_{r}^{t}E|D_{r}(Y_{s})|^{2}(t-r)dsdr\\ \nonumber
&\leq&K^{2}\epsilon\int_{t-\epsilon}^{t}\int_{r}^{t}E|D_{r}Y_{s}|^{2}dsdr \\ \nonumber
&=& K^{2}\epsilon\int_{t-\epsilon}^{t}\int_{t-\epsilon}^{s}E|D_{r}Y_{s}|^{2}drds \\ \nonumber
&\leq&K^{2}M\epsilon^{2}.
\end{eqnarray}
\begin{eqnarray*}
E\int_{t-\epsilon}^{t}[\int_{r}^{t}D_{r}(\sigma(Y_{s}))dB_{s}]^{2}dr
                         &\leq& \int_{t-\epsilon}^{t}E[\int_{r}^{t}D_{r}(\sigma(Y_{s}))^{2}ds]dr\\
                         &\leq& K^{2}\int_{t-\epsilon}^{t}E[\int_{r}^{t}(D_{r}Y_{s})^{2}ds]dr\\
                         &\leq& K^{2}\int_{t-\epsilon}^{t}\int_{t-\epsilon}^{s}E(D_{r}Y_{s})^{2}drds\\
                         &\leq& K^{2}\int_{t-\epsilon}^{t}\int_{s-\epsilon}^{s}E(D_{r}Y_{s})^{2}drds.
\end{eqnarray*}
Next we show that $\int_{s-\epsilon}^{s}E[(D_{r}Y_{s})]^{2}dr \leq C\epsilon$,
where $C$ is independent of $s$.
Because
$$
D_{r}Y_{s}^{n}
      =\sigma(Y_{r}^{n-1})+\int_{r}^{s}D_{r}(\sigma(Y_{u}^{n-1}))dB_{u}
       +\int_{r}^{s}D_{r}(b(Y_{u}^{n-1}))du+\frac{\alpha}{1-\alpha}D_{r}[\max\limits_{0\leq u\leq s}(Z_{u}^{n-1}+X_{u}^{n-1})],
$$
we have,
\begin{eqnarray}
E\int_{s-\epsilon}^{s}(D_{r}Y_{s}^{n})^{2}dr
     &\leq& 4E\int_{s-\epsilon}^{s}\sigma(Y_{r}^{n-1})^{2}dr \nonumber
         +4E\int_{s-\epsilon}^{s}[\int_{r}^{s}D_{r}(\sigma(Y_{u}^{n-1}))dB_{u}]^{2}dr\\ \nonumber
         &&+4E\int_{s-\epsilon}^{s}[\int_{r}^{s}D_{r}b(Y_{u}^{n-1})du]^{2}dr\\
         &&+4(\frac{\alpha}{1-\alpha})^{2}E\int_{s-\epsilon}^{s}(D_{r}(\max_{0\leq u\leq s}(Z_{u}^{n-1}+X_{u}^{n-1}))^{2}dr \label{bijiao1}\\
     &\leq& 4\int_{s-\epsilon}^{s}E[\sigma(Y_{r}^{n-1})^{2}]dr \nonumber
            +4\int_{s-\epsilon}^{s}E\int_{r}^{s}D_{r}(\sigma(Y_{u}^{n-1}))^{2}dudr\\ \nonumber
         && +4E\int_{s-\epsilon}^{s}[\int_{r}^{s}D_{r}(b(Y_{u}^{n-1}))^{2}du(s-r)]dr\\
         && +4(\frac{\alpha}{1-\alpha})^{2}E\sup_{0\leq u\leq s}\int_{s-\epsilon}^{s}[D_{r}(Z_{u}^{n-1}+X_{u}^{n-1})]^{2}dr \label{bijiao2}\\
     &\leq& 4\int_{s-\epsilon}^{s}E[\sigma(Y_{r}^{n-1})^{2}]dr \nonumber
           +4K^{2}\int_{s-\epsilon}^{s}\int_{r}^{s}E(D_{r}(Y_{u}^{n-1}))^{2}dudr\\ \nonumber
           &&+4K^{2}\epsilon \int_{s-\epsilon}^{s}\int_{r}^{s}E[D_{r}(Y_{u}^{n-1})]^{2}dudr\\ \nonumber
           &&+4(\frac{\alpha}{1-\alpha})^{2}E\sup_{0\leq u \leq s}\int_{s-\epsilon}^{s}[D_{r}(\int_{0}^{u}\sigma(Y_{v}^{n-1})dB_{v}+\int_{0}^{u}b(Y_{v}^{n-1})dv)]^{2}dr\\ \nonumber
     &\leq& 4\int_{s-\epsilon}^{s}E[\sigma(Y_{r}^{n-1})^{2}]dr
             +(4K^{2}+4K^{2}\epsilon)\int_{s-\epsilon}^{s}\int_{r}^{s}E(D_{r}Y_{u}^{n-1})^{2}dudr\\ \nonumber
           &&+4(\frac{\alpha}{1-\alpha})^{2}\{3\int_{s-\epsilon}^{s}E[\sigma(Y_{r}^{n-1})^{2}]dr
             +(3C_{1}K^{2}+3K^{2}\epsilon)\int_{s-\epsilon}^{s}\int_{r}^{s}E(D_{r}Y_{v}^{n-1})^{2}dvdr\}\\ \nonumber
     &=& (4+12(\frac{\alpha}{1-\alpha})^{2})\int_{s-\epsilon}^{s}E[\sigma(Y_{r}^{n-1})^{2}]dr\\ \nonumber
            &&+[4K^{2}+4K^{2}\epsilon+4(\frac{\alpha}{1-\alpha})^{2}(3C_{1}K^{2}
            +3K^{2}\epsilon)]\int_{s-\epsilon}^{s}\int_{r}^{s}E(D_{r}Y_{v}^{n-1})^{2}dvdr\\ \nonumber
     &=& C^{'}\int_{s-\epsilon}^{s}E[\sigma(Y_{r}^{n-1})^{2}]dr
            +C^{''}\int_{s-\epsilon}^{s}\int_{s-\epsilon}^{v}E(D_{r}Y_{v}^{n-1})drdv\\ \nonumber
     &\leq& C^{'}\int_{s-\epsilon}^{s}E[\sigma(Y_{r}^{n-1})^{2}]dr
            +C^{''}\int_{s-\epsilon}^{s}\int_{v-\epsilon}^{v}E(D_{r}(Y_{v}^{n-1}))^{2}drdv,
\end{eqnarray}
where we have used Proposition 2.1.10 in \cite{N} from (\ref{bijiao1}) to (\ref{bijiao2}).\\
Let $\phi_{n}(s)=E\int_{s-\epsilon}^{s}(D_{r}Y_{s}^{n})^{2}dr$, and $\phi(s)=E\int_{s-\epsilon}^{s}(D_{r}Y_{s})^{2}dr$,
then $\phi_{n}(s)\leq C^{*}\epsilon+C^{''}\int_{s-\epsilon}^{s}\phi_{n-1}(v)dv$.
Iterating it, we get
\begin{eqnarray}
\phi_{n}(s)
      &\leq& C^{*}\epsilon(1+C^{''}\epsilon+(C^{''}\epsilon)^{2}+...+(C^{''}\epsilon)^{n})\\
      &=&C^{*}\epsilon\frac{1}{1-C^{''}\epsilon}\\
      &\leq&2C^{*}\epsilon,
\end{eqnarray}
when $\epsilon$ is sufficiently small.
Then
\begin{eqnarray}
\phi(s)
  \leq \liminf_{n\rightarrow\infty}\phi_{n}(s)\leq 2C^{*}\epsilon.
\end{eqnarray}
So
\begin{eqnarray}
E\int_{t-\epsilon}^{t}[\int_{r}^{t}D_{r}(\sigma(Y_{s}))dB_{s}]^{2}dr
  \leq K^{2}\int_{t-\epsilon}^{t}\phi(s)ds
  \leq 2C^{*}K^{2}\epsilon^{2}.
\end{eqnarray}
Therefore
\begin{eqnarray}  \nonumber
  \lim_{\epsilon \rightarrow 0}\frac{1}{\epsilon}E\{\int_{t-\epsilon}^{t}([\int_{r}^{t}D_{r}(\sigma(Y_{s}))dB_{s}]^{2}
  +[\int_{r}^{t}D_{r}(b(Y_{s}))ds]^{2})dr\}=0.\\ \nonumber
\end{eqnarray}
Hence, there exists $\epsilon_{n} \downarrow 0,$
such that
\begin{eqnarray}
 \lim_{\epsilon_{n}\rightarrow 0}\frac{1}{\epsilon_{n}}\int_{t-\epsilon_{n}}^{t}([\int_{r}^{t}D_{r}(\sigma(Y_{s}))dB_{s}]^{2}
 +[\int_{r}^{t}D_{r}(b(Y_{s}))ds]^{2})dr=0
\ a.s..  \label{con*}
\end{eqnarray}
Set $$A_{n}=\{\omega: \max_{0 \leq s \leq t}Y_{s}(w)=\max_{0\leq s \leq t-\epsilon_{n}}Y_{s}(\omega)\},$$
and
\begin{eqnarray*}
A=\{\max_{0 \leq s \leq t}Y_{s}=Y_{t}\}.
\end{eqnarray*}
It is clear that $\Omega=\bigcup_{m=1}^{\infty}A_{m}\bigcup A.$\\
For $\omega \in A_{m}$, $\forall n>m$, we have
$$\int_{t-\epsilon_{n}}^{t}\alpha^{2}D_{r}(\max_{0\leq s \leq t-\epsilon_{m}}Y_{s}(\omega))^{2}dr=0.$$ \\
By the local property of the Malliavin derivative (Proposition 1.3.16 in \cite{N}) on $A_{m}$, we have \\
 \begin{eqnarray}\nonumber
 &&\lim_{n \rightarrow \infty}\frac{1}{\epsilon_{n}}\int_{t-\epsilon_{n}}^{t}\alpha^{2}(D_{r}(\max_{0 \leq s \leq t}Y_{s}(\omega))^{2}dr\\ \nonumber
 &=&\lim_{n \rightarrow \infty}\frac{1}{\epsilon_{n}}\int_{t-\epsilon_{n}}^{t}\alpha^{2}(D_{r}(\max_{0 \leq s \leq t-\epsilon_{m}}Y_{s}(\omega))^{2}dr \\ \label{con3}
 &=& 0.
 \end{eqnarray}
 Since $m$ is arbitrary,
by (\ref{con*}) and (\ref{con3}), we conclude that \\
$$
\lim_{\epsilon_{n}\rightarrow 0}\frac{1}{\epsilon_{n}}\int_{t-\epsilon_{n}}^{t}(D_{r}Y_{t})^{2}dr \geq \frac{1}{2}\sigma(Y_{t})^{2}>0 \ a.s. \mbox{ on} \bigcup_{m=1}^{\infty}A_{m}.
$$
For $\omega \in A$, according to (\ref{DYt}), we have
 \begin{eqnarray*}
(1-\alpha)D_{r}Y_{t}
&=&\sigma(Y_{r})+\int_{r}^{t}D_{r}(\sigma(Y_{s}))dB_{s}+\int_{r}^{t}D_{r}(b(Y_{s}))ds,
\end{eqnarray*}
\begin{eqnarray*}
(1-\alpha)^{2}(D_{r}Y_{t})^{2}
\geq \frac{1}{2}\sigma(Y_{r})^{2}-[\int_{r}^{t}D_{r}(\sigma(Y_{s}))dB_{s}+\int_{r}^{t}D_{r}(b(Y_{s}))ds]^{2},
\end{eqnarray*}
Since $\alpha <1$,
\begin{eqnarray*}
 (D_{r}Y_{t})^{2}
\geq \frac{1}{2(1-\alpha)^{2}}\sigma(Y_{r})^{2}-\frac{1}{(1-\alpha)^{2}}[\int_{r}^{t}D_{r}(\sigma(Y_{s}))dB_{s}+\int_{r}^{t}D_{r}(b(Y_{s}))ds]^{2}.
\end{eqnarray*}
Here on $A$,
\begin{eqnarray}
\lim_{\epsilon_{n} \rightarrow 0}\frac{1}{\epsilon_{n}}\int_{t-\epsilon_{n}}^{t}(D_{r}Y_{t})^{2}dr
\geq\frac{1}{2(1-\alpha)^{2}}\lim_{\epsilon_{n} \rightarrow 0}\frac{1}{\epsilon_{n}}\int_{t-\epsilon_{n}}^{t}\sigma(Y_{r})^{2}dr
=\frac{\sigma(Y_{t})^{2}}{2(1-\alpha)^{2}}>0.
\end{eqnarray}
Therefore
\begin{eqnarray}
||DY_{t}||_{H}^{2}=\int_{0}^{t}(D_{r}Y_{t})^{2}dr>0 \ \ a.s..
\end{eqnarray}
By Theorem 2.1, we conclude that the law of $Y_{t}$ is absolutely continuous with respect to Lebesgue measure.  $\Box$
\end{proof}\\

\section{Absolute continuity of the laws of perturbed reflected diffusion processes}
Consider the reflected, perturbed stochastic differential equation:
\begin{equation}
X_{t}=\int_{0}^{t}\sigma(X_{s})dB_{s}+\int_{0}^{t}b(X_{s})ds+\alpha \max_{0\leq s \leq t}X_{s}+L_{t}. \label{perturbed reflected eq}
\end{equation}

\begin{definition}
We say that $(X_{t},L_{t},t\geq 0)$ is a solution to the equation (\ref{perturbed reflected eq}) if\\
(i) $X_{0}=0, X_{t}\geq 0$ for $t\geq 0$,\\
(ii) $X_{t}, L_{t}$ are continuous and adapted to the filtration of $B$,\\
(iii) $L_{t}$ is non-decreasing with $L_{0}=0$ and $\int_{0}^{t}\chi\{X_{s}=0\}dL_{s}=L_{t}$,\\
(iv) $(X_{t},L_{t},t\geq 0)$ satisfies the equation (\ref{perturbed reflected eq}) almost surely for every $t>0$.
\end{definition}
we need the following lemma which strengthens the result of Proposition 2.1.10 in \cite{N}.
\begin{lemma}
Let $X=\{X_{s},0\leq s \leq t\}$ be a continuous process. Suppose that\\
(i) $E(\sup\limits_{0\leq s\leq t} X_{s}^{2})< \infty$,\\
(ii) for any $0\leq s\leq t, X_{s}\in \mathbb{D}^{1,2}$ and $E(\sup\limits_{0\leq s\leq t}||DX_{s}||_{H}^{2})<\infty$,\\
Then the random variable $M_{t}=\sup\limits_{0\leq s \leq t}X_{s}$ belongs to $\mathbb{D}^{1,2}$ and moreover,\\ $||DM_{t}||_{H}^{2}\leq \sup\limits_{0\leq s\leq t}||DX_{s}||_{H}^{2}$ a.s..
\end{lemma}

\begin{proof}
Consider a countable and dense subset $S_{0}=\{t_{n},n\geq 1\}$ of $[0,t]$.
Define $M_{n}=\sup\{X_{t_{1}},...,X_{t_{n}}\}$. The function $\varphi_{n}:\mathbb{R}^{n}\rightarrow\mathbb{R}$ defined by $\varphi_{n}(x_{1},...,x_{n})= max\{x_{1},...,x_{n}\}$ is Lipschitz. Therefore, we deduce that $M_{n}$ belongs to $\mathbb{D}^{1,2}$. Moreover, the sequence $M_{n}$ converges in $L^{2}(\Omega)$ to $M$. In order to evaluate the Malliavin derivative of $M_{n}$, we introduce the following sets:
\begin{eqnarray}
&&A_{1}=\{M_{n}=X_{t_{1}}\},\\ \nonumber
&& ... ...\\ \nonumber
&&A_{k}=\{M_{n}\neq X_{t_{1}},...,M_{n}\neq X_{t_{k-1}},M_{n}=X_{t_{k}}\},\ \  2 \leq k\leq n.
\end{eqnarray}
By the local property of the operator $D$, on the set $A_{k}$ the derivatives of the random variables $M_{n}$ and $X_{t_{k}}$ coincide. Hence, we can write
\begin{equation}
DM_{n}=\Sigma_{k=1}^{n}I_{A_{k}}DX_{t_{k}}
\end{equation}
Consequently,
\begin{equation}
E(||DM_{n}||_{H}^{2})\leq E(\sup_{0\leq s \leq t}||DX_{s}||_{H}^{2})< \infty
\end{equation}
Then, $M_{t}=\sup\limits_{0\leq s \leq t}X_{s}$ belongs to $\mathbb{D}^{1,2}$ and $DM_{n}$  weakly converges to $DM_{t}$ in $L^{2}(\Omega,P;H)$.\\
Now we want to show that
\begin{equation}
||DM_{t}||_{H}^{2}\leq \sup_{0\leq s \leq t}||DX_{s}||_{H}^{2}\ \   a.e..
\end{equation}
It is equivalent to prove that
for every non-negative bounded random variable $\xi$,
\begin{equation}
E[||DM_{t}||_{H}^{2}\xi]\leq E[\sup_{0\leq s \leq t}||DX_{s}||_{H}^{2}\xi], \label{eq:1}
\end{equation}
\begin{equation}
i.e.\int_{\Omega}||DM_{t}||_{H}^{2}\xi dP \leq \int_{\Omega}\sup_{0\leq s\leq t}||DX_{s}||_{H}^{2}\xi dP.
\end{equation}
Define $\mu(A)=\int_{A}\xi dP,\ \ \forall A \in \mathcal{F}$,
then (\ref{eq:1}) is equivalent to
\begin{equation}
\int_{\Omega}[||DM_{t}||_{H}^{2}]d\mu\leq \int_{\Omega}[\sup_{0 \leq s\leq t}||DX_{s}||_{H}^{2}]d\mu.
\end{equation}
For $ h \in L^{2}(\Omega,\mu;H)$, because $\xi$ is bounded, $\xi h\in L^{2}(\Omega, P; H)$.\\
Consequently, by the weak convergence of $DM_{n}$,
\begin{eqnarray}
\int_{\Omega}[(DM_{n},h)_{H}]d\mu\nonumber
&=&\int_{\Omega}(DM_{n},\xi h)_{H}dP \\ \nonumber
& \longrightarrow& \int_{\Omega}(DM_{t},\xi h)_{H}dP  \\ \nonumber
&=&\int_{\Omega}(DM_{t},h)d\mu \\ \nonumber
\end{eqnarray}
This shows that $DM_{n}\rightarrow DM_{t}$ weakly in $L^{2}(\Omega,\mu;H)$.\\
Hence, we have
\begin{equation}
\int_{\Omega}(||DM_{t}||_{H}^{2})d\mu \leq \liminf_{n\rightarrow \infty}\int_{\Omega}(||DM_{n}||_{H}^{2})d\mu\leq \int_{\Omega}(\sup_{0\leq s \leq t}||DX_{s}||_{H}^{2})d\mu< \infty. \ \Box \nonumber
\end{equation}

\end{proof}

\begin{theorem}
Assume $0\leq \alpha < \frac{1}{2}$. Let $(X_{t},L_{t},t\geq 0)$ be the unique solution to the equation (\ref{perturbed reflected eq}). Then $X_{t}$ belongs to $\mathbb{D}^{1,2}$ for any $t\geq 0$.
\end{theorem}

\begin{proof}
Consider the Picard iteration, $X_{t}^{0}=0, \  \forall t \in[0,T], \ T\geq 0$,
and let $(X_{t}^{n+1},L_{t}^{n+1})$ be the unique solution to the following reflected equation:
\begin{eqnarray}
X_{t}^{n+1}=\int_{0}^{t}\sigma(X_{s}^{n})dB_{s}+\int_{0}^{t}b(X_{s}^{n})ds+\alpha \max_{0\leq s \leq t}X_{s}^{n}+L_{t}^{n+1}. \label{interating eq}
\end{eqnarray}
By the reflection principle,
\begin{eqnarray}
L_{t}^{n+1}=-\inf_{s \leq t}\{(\int_{0}^{s}\sigma(X_{u}^{n})dB_{u}+\int_{0}^{s}b(X_{u}^{n})du+\alpha \max_{0 \leq u \leq s}X_{u}^{n})\wedge 0\}.\\ \nonumber
\end{eqnarray}
It was shown in \cite{RDZ}, there exists a unique solution $X_{t}$ to (\ref{perturbed reflected eq}). Next we are going to show that
\begin{equation}
\lim_{n\rightarrow\infty}E[\sup_{0\leq s\leq t}|X_{s}^{n}-X_{s}|^{2}]=0.
\end{equation}
 Now Eq(\ref{perturbed reflected eq}) and Eq(\ref{interating eq}) imply that:
 \begin{eqnarray} \nonumber
|X_{t}^{n+1}-X_{t}|
&\leq& |\int_{0}^{t}(\sigma(X_{s}^{n})-\sigma(X_{s}))dB_{s}|+2\alpha \max_{0 \leq s \leq t}|X_{s}^{n}-X_{s}|. \\ \nonumber
&&+|\int_{0}^{t}(b(X_{s}^{n})-b(X_{s}))ds|+\max_{0 \leq s \leq t}|\int_{0}^{s}(\sigma(X_{u}^{n})-\sigma(X_{u}))dB_{u}|\\ \nonumber
&&+\max_{0\leq s\leq t}|\int_{0}^{s}(b(X_{u}^{n})-b(X_{u}))du|, \\ \nonumber
\end{eqnarray}
where we have used the fact:
\begin{equation}
L_{t}=-\inf_{0\leq s\leq t}\{(\int_{0}^{s}\sigma(X_{u})dB_{u}+\int_{0}^{s}b(X_{u})du+\alpha \max_{0\leq u\leq s}X_{u})\wedge 0\}.
\end{equation}
Consequently,
\begin{eqnarray}
 \max_{0 \leq s \leq t}|X_{s}^{n+1}-X_{s}|
&\leq& 2\max_{0 \leq s \leq t}|\int_{0}^{s}(\sigma(X_{u}^{n})-\sigma(X_{u}))dB_{u}|\\ \nonumber
&&+2\max_{0\leq s\leq t}|\int_{0}^{s}(b(X_{u}^{n})-b(X_{u}))du|+2\alpha \max_{0 \leq s \leq t}|X_{s}^{n}-X_{s}|. \\ \nonumber
\end{eqnarray}
For any $\epsilon >0$, using the elementary inequality, $(a+b)^{2}\leq (1+C_{\epsilon})a^{2}+(1+\epsilon)b^{2},$
we obtain
\begin{eqnarray}
\max_{0 \leq s \leq t}|X_{s}^{n+1}-X_{s}|^{2} \nonumber
&\leq& 4(1+C_{\epsilon})[\max_{0 \leq s \leq t}|\int_{0}^{s}(\sigma(X_{u}^{n})-\sigma(X_{u}))dB_{u}|
 \\ \nonumber
 &&+\max_{0\leq s \leq t}|\int_{0}^{s}(b(X_{u}^{n})-b(X_{u}))du|]^{2}+(1+\epsilon)(2\alpha)^{2}\max_{0 \leq s \leq t}|X_{s}^{n}-X_{s}|^{2} \\ \nonumber
 &\leq& 8(1+C_{\epsilon})[\max_{0\leq s\leq t}|\int_{0}^{s}(\sigma(X_{u}^{n})-\sigma(X_{u}))dB_{u}|^{2}
 \\ \nonumber
 &&+\max_{0\leq s\leq t}|\int_{0}^{s}(b(X_{u}^{n})-b(X_{u}))du|^{2}]+(1+\epsilon)(2\alpha)^{2}\max_{0\leq s\leq t}|X_{s}^{n}-X_{s}|^{2}.
 \end{eqnarray}
 By Burkh\"{o}lder inequality,
 \begin{eqnarray}
 E[\max_{0 \leq s \leq t}|X_{s}^{n+1}-X_{s}|^{2}] \nonumber
 &\leq& 8(1+C_{\epsilon})\{E[\max_{0 \leq s \leq t}|\int_{0}^{s}(\sigma(X_{u}^{n})-\sigma(X_{u}))dB_{u}|^{2}]\\ \nonumber
 &&+E[\max_{0\leq s\leq t}|\int_{0}^{s}(b(X_{u}^{n})-b(X_{u}))du|^{2}]\}\\ \nonumber
 &&+(1+\epsilon)(2\alpha)^{2}E[\max_{0 \leq s \leq t}|X_{s}^{n}-X_{s}|^{2}]  \\ \nonumber
 &\leq& 8(1+C_{\epsilon})(K_{1}C^{2}+TC^{2})E[\int_{0}^{t}|X_{u}^{n}-X_{u}|^{2}du]\\ \nonumber
&&+(1+\epsilon)(2\alpha)^{2}E[\max_{0 \leq s \leq t}|X_{s}^{n}-X_{s}|^{2}].
 \end{eqnarray}
 Let $g_{n+1}(t)=E[\max\limits_{0 \leq s \leq t}|X_{s}^{n+1}-X_{s}|^{2}]$. The above inequality implies
\begin{eqnarray}
g_{n+1}(t)\leq 8(K_{1}C^{2}+TC^{2})(1+C_{\epsilon})\int_{0}^{t}g_{n}(s)ds+(1+\epsilon)(2\alpha)^{2}g_{n}(t).\\ \nonumber
\end{eqnarray}
Summing the above equations from 1 to M:
\begin{eqnarray}
\sum_{n=1}^{M}g_{n+1}(t)\leq 8(K_{1}C^{2}+TC^{2})(1+C_{\epsilon})\int_{0}^{t}\sum_{n=1}^{M}g_{n}(s)ds+(1+\epsilon)(2\alpha)^{2}\sum_{n=1}^{M}g_{n}(t).
\end{eqnarray}
And then,
\begin{eqnarray}
\sum_{n=1}^{M}g_{n}(t)-g_{1}(t)
&\leq& \sum_{n=1}^{M}g_{n+1}(t) \\
&\leq& C^{*}\int_{0}^{t}\sum_{n=1}^{M}g_{n}(s)ds+ \beta \sum_{n=1}^{M}g_{n}(t), \label{gn}
\end{eqnarray}
where $\beta=(1+\epsilon)(2\alpha)^{2}$, $C^{*}$ is a constant.
Choose $\epsilon >0$ sufficiently small so that $\beta=(1+\epsilon)(2\alpha)^{2}<1.$\\
It follows from (\ref{gn}) that
 \begin{eqnarray}
 (1-\beta)\sum_{n=1}^{M}g_{n}(t)
 \leq g_{1}(t)+C^{*}\int_{0}^{t}\sum_{n=1}^{M}g_{n}(s)ds.
\end{eqnarray}
By Gronwall inequality,
\begin{eqnarray}
\sum_{n=1}^{M}g_{n}(t)\leq \frac{g_{1}(T)}{1-\beta}e^{\frac{C^{*}}{1-\beta} T}.
\end{eqnarray}
Let $M\rightarrow\infty$ to get
\begin{eqnarray}
\sum_{n=1}^{\infty}E[\max_{0\leq s\leq t}|X_{s}^{n}-X_{s}|^{2}]< \infty.
\end{eqnarray}
which yields that $X_{t}^{n}$ converges to $X_{t}$ in $L^{2}(\Omega).$\\
Let $Y_{s}^{n}=\max\limits_{0\leq u\leq s}X_{u}^{n}.$
We will prove the following property by induction on n.\\
(\textbf{P}). $X_{t}^{n}\in \mathbb{D}^{1,2}$, $E(\max\limits_{0\leq s\leq t}||DX_{s}^{n}||_{H}^{2})< \infty$, $E(\max\limits_{0\leq s\leq t}||DY_{s}^{n}||_{H}^{2})<\infty$.\\
Clearly, (\textbf{P}) holds for $n=0$.\\
Suppose (\textbf{P})holds for $n$. We prove that it is valid for $n+1$.\\
Now we note that
\begin{eqnarray}
\int_{0}^{t}\sigma(X_{s}^{n})dB_{s}\in \mathbb{D}^{1,2},
\int_{0}^{t}b(X_{s}^{n})ds\in\mathbb{D}^{1,2}, \label{the two elements}
\end{eqnarray}
and
\begin{eqnarray*}
D_{r}(\int_{0}^{t}\sigma(X_{s}^{n})dB_{s})=\sigma(X_{r}^{n})+\int_{r}^{t}D_{r}(\sigma(X_{s}^{n}))dB_{s},\\
D_{r}(\int_{0}^{t}b(X_{s}^{n})ds)=\int_{r}^{t}D_{r}(b(X_{s}^{n}))ds.
\end{eqnarray*}
Next we prove $ \max\limits_{0 \leq s \leq t}X_{s}^{n}\in \mathbb{D}^{1,2}.$\\
As
\begin{eqnarray}
\max_{0\leq s \leq t}|X_{s}^{n}|\nonumber
&\leq& 2 \max_{0\leq s \leq t}|\int_{0}^{s}\sigma(X_{u}^{n-1})dB_{u}|+2\alpha \max_{0\leq s \leq t}|X_{s}^{n-1}|+2\max_{0\leq s\leq t}|\int_{0}^{s}b(X_{u}^{n-1})du|, \nonumber
\end{eqnarray}
we get
\begin{eqnarray}
E(\max_{0\leq s \leq t}|X_{s}^{n}|^{2})\nonumber
&\leq& 12E[\max_{0\leq s \leq t}|\int_{0}^{s}\sigma(X_{u}^{n-1})dB_{u}|^{2}]+12\alpha^{2}E(\max_{0\leq s \leq t}|X_{s}^{n-1}|^{2})\\ \nonumber
&&+12E[\max_{0\leq s \leq t}|\int_{0}^{s}b(X_{u}^{n-1})du|^{2}]\\ \nonumber
&\leq& 12K_{1}E[\int_{0}^{t}\sigma(X_{u}^{n-1})^{2}du]+12\alpha^{2}E(\max_{0\leq s \leq t}|X_{s}^{n-1}|^{2})+12TE[\int_{0}^{t}b(X_{u}^{n-1})^{2}du]\\ \nonumber
&\leq& 12(K_{1}+T)C^{2}[E\int_{0}^{t}(1+(X_{u}^{n-1})^{2})du]+12\alpha^{2}E(\max_{0\leq s \leq t}|X_{s}^{n-1}|^{2})\\ \nonumber
&\leq& 12(K_{1}+T)C^{2}T+[12(K_{1}+T)C^{2}T+12\alpha^{2}]E(\max_{0\leq s \leq t}|X_{s}^{n-1}|^{2}).
\end{eqnarray}
By interation, we see that
\begin{equation}
E(\max_{0\leq s \leq t}|X_{s}^{n}|^{2})< \infty.
\end{equation}
By the induction hypothesis $E(\max\limits_{0\leq s\leq t}||DX_{s}^{n}||_{H}^{2})< \infty$ and Proposition 2.1.10 in \cite{N}, it follows that
\begin{equation}
\max_{0\leq s \leq t}X_{s}^{n}\in \mathbb{D}^{1,2}. \label{firstly}
\end{equation}
Now we want to show
\begin{equation}
L_{t}^{n+1}=\max_{0\leq s \leq t}\{-(\int_{0}^{s}\sigma(X_{u}^{n})dB_{u}+\int_{0}^{s}b(X_{u}^{n})du+\alpha \max_{0\leq u \leq s}X_{u}^{n})\vee 0\}\in \mathbb{D}^{1,2}.
\end{equation}
Let \begin{equation}
V_{s}^{n}:=-(\int_{0}^{s}\sigma(X_{u}^{n})dB_{u}+\int_{0}^{s}b(X_{u}^{n})du+\alpha \max_{0\leq u \leq s}X_{u}^{n})\vee 0 .
\end{equation}
Firstly, $V_{s}^{n}\in \mathbb{D}^{1,2}$ by (\ref{the two elements}) and (\ref{firstly}).
Secondly,
\begin{eqnarray}
E(\max_{0\leq s \leq t}(V_{s}^{n})^{2})\nonumber
&\leq& 3E[\max_{0\leq s \leq t}(\int_{0}^{s}\sigma(X_{u}^{n})dB_{u})^{2}]+3\alpha^{2}E[\max_{0\leq s \leq t}(X_{s}^{n})^{2}]+3E[\max_{0\leq s\leq t}(\int_{0}^{s}b(X_{u}^{n})du)^{2}]\\ \nonumber
&=& 3K_{1}E[\int_{0}^{t}\sigma(X_{u}^{n})^{2}du]+3\alpha^{2}E[\max_{0\leq s \leq t}(X_{s}^{n})^{2}]+3TE[\int_{0}^{t}b(X_{s}^{n})^{2}ds]\\ \nonumber
&\leq& 3(K_{1}+T)C^{2}E\int_{0}^{t}(1+(X_{u}^{n})^{2})du+3\alpha^{2}E[\max_{0\leq s\leq t}(X_{s}^{n})^{2}]\\ \nonumber
&\leq&3C^{2}(K_{1}+T)(T+TE[\max_{0\leq s\leq t}(X_{s}^{n})^{2}])+3\alpha^{2}E[\max_{0\leq s \leq t}(X_{s}^{n})^{2}]\\\nonumber
&=&3C^{2}T(K_{1}+T)+[3C^{2}(K_{1}+T)T+3\alpha^{2}]E[\max_{0\leq s\leq t}(X_{s}^{n})^{2}] \\
&<& \infty. \label{secondly}
\end{eqnarray}
Thirdly,
\begin{eqnarray}
E(\max_{0\leq s\leq t}||DV_{s}^{n}||_{H}^{2})\nonumber
&=&E(\max_{0\leq s \leq t}\int_{0}^{s}(D_{r}(V_{s}^{n}))^{2}dr)\\ \nonumber
&\leq&3(1+C_{\epsilon})E\int_{0}^{t}\sigma(X_{r}^{n})^{2}dr+3(1+C_{\epsilon})E[\max_{0\leq s \leq t}\int_{0}^{s}(\int_{r}^{s}D_{r}(\sigma(X_{u}^{n}))dB_{u})^{2}dr]\\ \nonumber
&&+3(1+C_{\epsilon})E[\max_{0\leq s\leq t}\int_{0}^{s}(\int_{r}^{s}D_{r}(b(X_{u}^{n}))du)^{2}dr] \\ \nonumber
&&+(1+\epsilon)\alpha^{2}E[\max_{0\leq s \leq t}\int_{0}^{s}(D_{r}(Y_{s}^{n}))^{2}dr]\\ \nonumber
&\leq&3(1+C_{\epsilon})E\int_{0}^{t}\sigma(X_{r}^{n})^{2}dr+3(1+C_{\epsilon})\int_{0}^{t}[E\int_{r}^{t}(D_{r}(\sigma(X_{u}^{n})))^{2}du]dr\\ \nonumber
&&+3(1+C_{\epsilon})t\int_{0}^{t}E[\int_{r}^{t}(D_{r}b(X_{u}^{n}))^{2}du]dr+(1+\epsilon)\alpha^{2}E[\max_{0\leq s \leq t}||DY_{s}^{n}||_{H}^{2}]\\ \nonumber
&\leq& 3(1+C_{\epsilon})E\int_{0}^{t}\sigma(X_{r}^{n})^{2}dr+3(1+C_{\epsilon})(1+t)K^{2}\int_{0}^{t}E||DX_{u}^{n}||_{H}^{2}du\\ \nonumber
&&+(1+\epsilon)\alpha^{2}E[\max_{0\leq s \leq t}||DY_{s}^{n}||_{H}^{2}] \\
&<&\infty. \label{thirdly}
\end{eqnarray}
By Proposition 2.1.10 in \cite{N}, (\ref{secondly}) and (\ref{thirdly}) yield that
$L_{t}^{n+1}\in \mathbb{D}^{1,2}.$
Thus, we conclude $X_{t}^{n+1}\in \mathbb{D}^{1,2}.$\\
Moreover,
\begin{eqnarray}
D_{r}(X_{s}^{n+1})\nonumber
&=&\sigma(X_{r}^{n})+\int_{r}^{s}D_{r}(\sigma(X_{u}^{n}))dB_{u}+\int_{r}^{s}D_{r}(b(X_{u}^{n}))du+\alpha D_{r}(\max_{0\leq u \leq s }X_{u}^{n})+D_{r}(L_{s}^{n+1}), \nonumber
\end{eqnarray}
and
\begin{eqnarray}
||D(X_{s}^{n+1})||_{H}^{2}
&=&\int_{0}^{s}(D_{r}(X_{s}^{n+1}))^{2}dr\\ \nonumber
&\leq&5\int_{0}^{s}\sigma(X_{r}^{n})^{2}dr+5\int_{0}^{s}[\int_{r}^{s}D_{r}(\sigma(X_{u}^{n}))dB_{u}]^{2}dr
+5\int_{0}^{s}[\int_{r}^{s}(D_{r}(b(X_{u}^{n})))du]^{2}dr\\ \nonumber
&&+5\alpha^{2}||DY_{s}^{n}||_{H}^{2}+5||DL_{s}^{n+1}||_{H}^{2}.
\end{eqnarray}
So
\begin{eqnarray}
E[\max_{0\leq s \leq t}||DX_{s}^{n+1}||_{H}^{2}]\nonumber
              &\leq& 5E[\int_{0}^{t}\sigma(X_{r}^{n})^{2}dr]+5E\int_{0}^{t}\max_{0\leq s\leq t}[\int_{r}^{s}D_{r}(\sigma(X_{u}^{n}))dB_{u}]^{2}dr \\ \nonumber
                                & &+5TE\int_{0}^{t}\int_{r}^{t}(D_{r}(b(X_{u}^{n})))^{2}dudr+5\alpha^{2}E[\max_{0\leq s \leq t}||DY_{s}^{n}||_{H}^{2}]\\ \nonumber
&&+5E[\max_{0\leq s \leq t}||DL_{s}^{n+1}||_{H}^{2}]\\ \nonumber
&\leq& 5E\int_{0}^{t}\sigma(X_{r}^{n})^{2}dr+5K_{1}\int_{0}^{t}E\int_{r}^{t}(D_{r}(\sigma(X_{u}^{n})))^{2}dudr
 \\ \nonumber
&&+5TE\int_{0}^{t}\int_{r}^{t}(D_{r}(b(X_{u}^{n})))^{2}dudr+5\alpha^{2}E[\max_{0\leq s\leq t}||DY_{s}^{n}||_{H}^{2}]\\ \nonumber
&&+5E[\max_{0\leq s \leq t}||DL_{s}^{n+1}||_{H}^{2}]\\ \nonumber
&\leq& 5E\int_{0}^{t}\sigma(X_{r}^{n})^{2}dr+(5K_{1}K^{2}+5TK^{2})\int_{0}^{t}E||DX_{u}^{n}||_{H}^{2}du\\ \nonumber
&&+5\alpha^{2}E[\max_{0\leq s \leq t}||DY_{s}^{n}||_{H}^{2}]+5E[\max_{0\leq s \leq t}||DL_{s}^{n+1}||_{H}^{2}].
\end{eqnarray}
To prove
\begin{equation}
E[\max_{0\leq s \leq t}||DX_{s}^{n+1}||_{H}^{2}]< \infty,
\end{equation}
we only need to prove
\begin{equation}
E[\max_{0\leq s \leq t}||DL_{s}^{n+1}||_{H}^{2}]<\infty.
\end{equation}
According to Lemma 4.1,
\begin{equation}
||DL_{s}^{n+1}||_{H}^{2}\leq
\sup_{0\leq u \leq s}||DV_{u}^{n}||_{H}^{2}.
\end{equation}
Thus we have
\begin{eqnarray}
\max_{0\leq s\leq t}||DL_{s}^{n+1}||_{H}^{2}\leq \max_{0\leq s\leq t}(\sup_{0\leq u \leq s} ||DV_{u}^{n}||_{H}^{2})=\max_{0\leq s \leq t}||DV_{s}^{n}||_{H}^{2},
\end{eqnarray}
and by (\ref{thirdly}),
\begin{eqnarray}
E[\max_{0\leq s \leq t}||DL_{s}^{n+1}||_{H}^{2}]\leq E[\sup_{0\leq s \leq t}||DV_{s}^{n}||_{H}^{2}]< \infty.
\end{eqnarray}
Again by Lemma 4.1,
\begin{equation}
||DY_{s}^{n+1}||_{H}^{2}\leq \sup_{0\leq u \leq s}||DX_{u}^{n+1}||_{H}^{2}.
\end{equation}
Hence,
\begin{eqnarray}
E[\max_{0\leq s \leq t}||DY_{s}^{n+1}||_{H}^{2}] \leq E[\sup_{0\leq s\leq t}||DX_{s}^{n+1}||_{H}^{2}]< \infty.
\end{eqnarray}
We've proved property (\textbf{P}).\\
Next we prove
\begin{equation}
\sup_{n}E||DX_{t}^{n+1}||_{H}^{2}<\infty.
\end{equation}
Because, for any $\epsilon >0,$
\begin{eqnarray*}
|D_{r}X_{s}^{n+1}|^{2}
&\leq & (1+C_{\epsilon})[3\sigma(X_{r}^{n})^{2}+3(\int_{r}^{s}D_{r}(\sigma(X_{u}^{n}))dB_{u})^{2}
+3(\int_{r}^{s}D_{r}(b(X_{u}^{n}))du)^{2}]\\
&&+(1+\epsilon)[2\alpha^{2}D_{r}(\max_{0\leq u\leq s}X_{u}^{n})^{2}+2D_{r}(L_{s}^{n+1})^{2}].
\end{eqnarray*}
We have
\begin{eqnarray*}
||DX_{s}^{n+1}||_{H}^{2}
                  &\leq& 3(1+C_{\epsilon})\int_{0}^{s}\sigma(X_{r}^{n})^{2}dr
                          +3(1+C_{\epsilon})\int_{0}^{s}[\int_{r}^{s}D_{r}(\sigma(X_{u}^{n}))dB_{u}]^{2}dr \\
                       & &+3(1+C_{\epsilon})\int_{0}^{s}[\int_{r}^{s}D_{r}(b(X_{u}^{n}))du]^{2}dr\\
                       & &+2(1+\epsilon)\alpha^{2}\int_{0}^{s}D_{r}(\max_{0\leq u\leq s}X_{u}^{n})^{2}dr
                          +2(1+\epsilon)\int_{0}^{s}D_{r}(L_{s}^{n+1})^{2}dr\\
                   &=& 3(1+C_{\epsilon})\int_{0}^{s}\sigma(X_{r}^{n})^{2}dr
                          +3(1+C_{\epsilon})\int_{0}^{s}[\int_{r}^{s}D_{r}(\sigma(X_{u}^{n}))dB_{u}]^{2}dr\\
                       & &+3(1+C_{\epsilon})\int_{0}^{s}[\int_{r}^{s}D_{r}(b(X_{u}^{n}))du]^{2}dr\\
                       & &+2(1+\epsilon)\alpha^{2}||DY_{s}^{n}||_{H}^{2}+2(1+\epsilon)||DL_{s}^{n+1}||_{H}^{2}\\
                  &\leq& 3(1+C_{\epsilon})\int_{0}^{s}\sigma(X_{r}^{n})^{2}dr
                          +3(1+C_{\epsilon})\int_{0}^{s}[\int_{r}^{s}D_{r}\sigma(X_{u}^{n})dB_{u}]^{2}dr\\
                       & &+3(1+C_{\epsilon})\int_{0}^{s}[\int_{r}^{s}D_{r}(b(X_{u}^{n}))du]^{2}dr\\
                       & &+2(1+\epsilon)\alpha^{2}\sup_{0\leq u\leq s}||DX_{u}^{n}||_{H}^{2}
                          +2(1+\epsilon)\sup_{0\leq u\leq s}||DV_{u}^{n}||_{H}^{2},
\end{eqnarray*}
 where Lemma 4.1 was used in the last step.
Hence, using Ito's Isometry we have
\begin{eqnarray}
E(\sup_{0\leq s\leq t}||DX_{s}^{n+1}||_{H}^{2}) \nonumber
&\leq& 3(1+C_{\epsilon})E\int_{0}^{t}\sigma(X_{r}^{n})^{2}dr
+3K_{1}K^{2}(1+C_{\epsilon})\int_{0}^{t}E\int_{r}^{t}(D_{r}(X_{u}^{n}))^{2}dudr\\ \nonumber
&&+3TK^{2}(1+C_{\epsilon})\int_{0}^{t}E\int_{r}^{t}(D_{r}(X_{u}^{n}))^{2}dudr\\ \nonumber
&&+2(1+\epsilon)\alpha^{2}E[\sup_{0\leq u \leq t}||DX_{u}^{n}||_{H}^{2}]
+2(1+\epsilon)E[\sup_{0\leq u\leq t}||DV_{u}^{n}||_{H}^{2}]\\ \nonumber
&=&3(1+C_{\epsilon})E\int_{0}^{t}\sigma(X_{r}^{n})^{2}dr
+3(K_{1}+T)K^{2}(1+C_{\epsilon})\int_{0}^{t}E||DX_{u}^{n}||_{H}^{2}du\\ \nonumber
&&+2(1+\epsilon)\alpha^{2}E\sup_{0\leq s\leq t}||DX_{s}^{n}||_{H}^{2}+2(1+\epsilon)\{3(1+C_{\epsilon})E\int_{0}^{t}\sigma(X_{r}^{n})^{2}dr\\ \nonumber
&&+6(1+C_{\epsilon})K^{2}\int_{0}^{t}E||DX_{u}^{n}||_{H}^{2}du+(1+\epsilon)\alpha^{2}E[\sup_{0\leq s\leq t}||DY_{s}^{n}||_{H}^{2}]\}\\ \nonumber
&\leq& 3(1+C_{\epsilon})\int_{0}^{t}E[\sigma(X_{r}^{n})^{2}]dr+3(K_{1}+T)K^{2}(1+C_{\epsilon})\int_{0}^{t}E||DX_{u}^{n}||_{H}^{2}du\\ \nonumber
&&+2(1+\epsilon)\alpha^{2}E[\sup_{0\leq s\leq t}||DX_{s}^{n}||_{H}^{2}]+6(1+\epsilon)(1+C_{\epsilon})E\int_{0}^{t}\sigma(X_{r}^{n})^{2}dr\\ \nonumber
&&+12(1+\epsilon)(1+C_{\epsilon})K^{2}\int_{0}^{t}E||DX_{u}^{n}||_{H}^{2}du+(1+\epsilon)^{2}\alpha^{2}E[\sup_{0\leq s \leq t}||DX_{s}^{n}||_{H}^{2}]\\ \nonumber
&=&(9+6\epsilon)(1+C_{\epsilon})\int_{0}^{t}E[\sigma(X_{r}^{n})^{2}]dr+2(2+\epsilon)(1+\epsilon)\alpha^{2}E(\sup_{0\leq s\leq t}||DX_{s}^{n}||_{H}^{2})\\
&&+[12K^{2}(1+\epsilon)(1+C_{\epsilon})+3K^{2}(K_{1}+T)(1+C_{\epsilon})]\int_{0}^{t}E||DX_{u}^{n}||_{H}^{2}du. \label{ineq5}
\end{eqnarray}
Note that $\sup_{n}\int_{0}^{t}E[\sigma(X_{r}^{n})^{2}]dr\leq C\sup_{n}\int_{0}^{t}E(1+|X_{r}^{n}|^{2})dr<\infty$.\\
Let
\begin{eqnarray*}
\psi_{n}(t)=E(\sup_{0\leq s\leq t}||DX_{s}^{n}||_{H}^{2}).
\end{eqnarray*}
Then from (\ref{ineq5}), we have
\begin{eqnarray*}
\psi_{n+1}(t)\leq c_{1}+c_{2}\psi_{n}(t)+c_{3}\int_{0}^{t}\psi_{n}(u)du,
\end{eqnarray*}
where $c_{2}=2(2+\epsilon)(1+\epsilon)\alpha^{2}<1$ when $\epsilon >0$ is sufficiently small, according to $\alpha<\frac{1}{2}$.\\
Iterating this inequality, we obtain
$$\sup_{n}\psi_{n+1}(t)<\infty, \ \mbox{i.e.}\ \  \sup_{n}E[\max_{0\leq s \leq t}||DX_{s}^{n+1}||_{H}^{2}]< \infty.$$
According to Lemma 1.2.3 in \cite{N}, $X_{t} \in \mathbb{D}^{1,2}$.  $\Box$
\end{proof}

To study the absolute continuity of the law, we need the following comparison theorem.

\begin{lemma}
Assume $0\leq \alpha < \frac{1}{2}$. Let $X_{t}$ be the solution to the perturbed, reflected stochastic differential equation
$$
X_{t}=\int_{0}^{t}\sigma(X_{s})dB_{s}+\int_{0}^{t}b(X_{s})ds+\alpha \max\limits_{0\leq s\leq t}X_{s}+L_{t}.
$$
Let $Y_{t}$ be the solution to the reflected stochastic equation $Y_{t}=\int_{0}^{t}\sigma(Y_{s})dB_{s}+\int_{0}^{t}b(Y_{s})ds+\widetilde{L_{t}}$.
Then, we have that $Y_{t}\leq X_{t}$ a.e..
\end{lemma}

\begin{proof}\\
Let $\Delta_{t}=Y_{t}-X_{t}=\widetilde{L_{t}}-L_{t}+\int_{0}^{t}(b(Y_{s})-b(X_{s}))ds
+\int_{0}^{t}(\sigma(Y_{s})-\sigma(X_{s}))dB_{s}-\alpha \max\limits_{0\leq s\leq t}X_{s}.$\\
There exists a strictly decreasing sequence $\{a_{n}\}_{n=0}^{\infty} \subseteq (0,1]$ with $a_{0}=1$, $\lim_{n\rightarrow \infty}a_{n}=0$ and $\int_{a_{n}}^{a_{n-1}}\frac{1}{c^{2}u^{2}}du=n$, for every $n\geq1$. For each $n\geq1$, there exists a continuous function $\rho_{n}$ on $R$ with support in $(a_{n},a_{n-1})$ so that $0\leq \rho_{n}(x)\leq \frac{2}{nC^{2}x^{2}}$ holds for every $x>0$, and $\int_{a_{n}}^{a_{n-1}}\rho_{n}(x)dx=1$. Then the function
\begin{eqnarray*}
\phi_{n}(x)=\int_{0}^{|x|}\int_{0}^{y}\rho_{n}(u)dudy I_{(0,\infty)}(x), x\in R.
\end{eqnarray*}
is twice continuously differentiable, with
$0\leq \phi_{n}^{'}(x)\leq 1$ and $\lim\limits_{n\rightarrow\infty}\phi_{n}(x)=x^{+}$ for $x\in R.$\\
By the Ito rule:
\begin{eqnarray*}
\phi_{n}(\Delta_{t})
                 &=&\int_{0}^{t}\phi_{n}^{'}(\Delta_{s})d\widetilde{L_{s}}-\int_{0}^{t}\phi_{n}^{'}(\Delta_{s})dL_{s}
                      -\alpha\int_{0}^{t}\phi_{n}^{'}(\Delta_{s})d(\max_{0\leq u\leq s}X_{u})\\
                   &&+\int_{0}^{t}\phi_{n}^{'}(\Delta_{s})(b(Y_{s})-b(X_{s}))ds
                      +\int_{0}^{t}\phi_{n}^{'}(\Delta_{s})(\sigma(Y_{s})-\sigma(X_{s}))dB_{s}\\
                   && +\frac{1}{2}\int_{0}^{t}\phi_{n}^{''}(\Delta_{s})(\sigma(Y_{s})-\sigma(X_{s}))^{2}ds \\ \nonumber
                &\leq& \int_{0}^{t}\phi_{n}^{'}(\Delta_{s})d\widetilde{L_{s}}
                     +C\int_{0}^{t}\phi_{n}^{'}(\Delta_{s})I_{\{Y_{s}>X_{s}\}}|Y_{s}-X_{s}|ds \\
                   &&+\int_{0}^{t}\phi_{n}^{'}(\Delta_{s})(\sigma(Y_{s})-\sigma(X_{s}))dB_{s}\\
                    && +\frac{1}{2}\int_{0}^{t}\phi_{n}^{''}(\Delta_{s})(\sigma(Y_{s})-\sigma(X_{s}))^{2}ds
\end{eqnarray*}
Hence,
\begin{eqnarray*}
E[\phi_{n}(\Delta_{t})]
      &\leq & E\int_{0}^{t}\phi_{n}^{'}(\Delta_{s})\chi_{\{Y_{s}>0\}}d\widetilde{L_{s}}+CE\int_{0}^{t}(Y_{s}-X_{s})^{+}ds\\
      &&+\frac{1}{2}E\int_{0}^{t}\phi_{n}^{''}(\Delta_{s})(\sigma(Y_{s})-\sigma(X_{s}))^{2}ds \\
      &\leq& C\int_{0}^{t}E(Y_{s}-X_{s})^{+}ds+ \frac{t}{n}
\end{eqnarray*}
Letting $n\rightarrow \infty,$ we get $E\Delta_{t}^{+} \leq C\int_{0}^{t}E\Delta_{s}^{+}ds$.
By Gronwall Inequality, $E\Delta_{t}^{+}=0$. Hence $Y_{t}\leq X_{t}$ a.e.. $\Box$
\end{proof}

\begin{theorem}
Assume $0\leq \alpha < \frac{1}{2}$. Let $X_{t}$ be the solution to the equation (\ref{perturbed reflected eq}). Suppose that $\sigma(\cdot)$ and $b(\cdot)$ are Lipschitz continuous and $|\sigma(x)|>0$ for $x \in R$. Then for $t>0$, the law of $X_{t}$ is absolutely continuous with respect to Lebesgue measure.
\end{theorem}

\begin{proof}
It is sufficient to prove $||DX_{t}||_{H}^{2}>0$ a.s. according to Theorem 2.1. \\
Now,$$
X_{t}=\int_{0}^{t}\sigma(X_{s})dB_{s}+\int_{0}^{t}b(X_{s})ds+\alpha \max_{0\leq s\leq t}X_{s}+L_{t},$$
Let$$
V_{s}={-(\int_{0}^{s}\sigma(X_{u})dB_{u}+\int_{0}^{s}b(X_{u})du+\alpha \max_{0\leq u\leq s}X_{u})\vee 0}.$$
Then, by reflection principle,
\begin{eqnarray*}
L_{t}=\max_{0\leq s\leq t}[{-(\int_{0}^{s}\sigma(X_{u})dB_{u}+\int_{0}^{s}b(X_{u})du+\alpha \max_{0\leq u\leq s}X_{u})\vee 0}]=\max_{0\leq s \leq t}V_{s},\\
D_{r}X_{t}=\sigma(X_{r})+\int_{r}^{t}D_{r}(\sigma(X_{s}))dB_{s}+\int_{r}^{t}D_{r}(b(X_{s}))ds+\alpha D_{r}(\max_{0\leq s \leq t}X_{s})+D_{r}(\max_{0\leq s\leq t}V_{s}).\\
(D_{r}X_{t})^{2}\geq \frac{1}{2}\sigma(X_{r})^{2}-[\int_{r}^{t}D_{r}(\sigma(X_{s}))dB_{s}+\int_{r}^{t}D_{r}(b(X_{s}))ds+\alpha D_{r}(\max_{0\leq s \leq t}X_{s})+D_{r}(\max_{0\leq s\leq t}V_{s})]^{2}
\end{eqnarray*}
Similar as in Section 3, we have
\begin{eqnarray}  \nonumber
  \lim_{\epsilon \rightarrow 0}\frac{1}{\epsilon}E\{\int_{t-\epsilon}^{t}([\int_{r}^{t}D_{r}(\sigma(X_{s}))dB_{s}]^{2}+[\int_{r}^{t}D_{r}(b(X_{s}))ds]^{2})dr\}=0.\\ \nonumber
\end{eqnarray}
Hence, there exists $\epsilon_{n} \downarrow 0,$
such that
\begin{eqnarray}
 \lim_{\epsilon_{n}\rightarrow 0}\frac{1}{\epsilon_{n}}\int_{t-\epsilon_{n}}^{t}([\int_{r}^{t}D_{r}(\sigma(X_{s}))dB_{s}]^{2}+[\int_{r}^{t}D_{r}(b(X_{s}))ds]^{2})dr=0
\ a.s. . \label{con2}
\end{eqnarray}
Let
\begin{eqnarray*}
A_{n}=\{\omega:\max_{0\leq s\leq t}X_{s}=\max_{0\leq s\leq t-\epsilon_{n}}X_{s}\},\\
A=\{\omega:\max_{0\leq s\leq t}X_{s}=X_{t}\}.
\end{eqnarray*}
Then,
\begin{equation}
\Omega=\cup_{m=1}^{\infty}A_{m}\cup A.
\end{equation}
Let
\begin{eqnarray*}
B_{n}=\{\omega: \max_{0\leq s\leq t}V_{s}=\max_{0\leq s \leq t-\epsilon_{n}}V_{s}\},\\
B=\{\omega:\max_{0\leq s\leq t}V_{s}=V_{t}\}.
\end{eqnarray*}
We have,
\begin{equation}
\Omega=\cup_{n=1}^{\infty}B_{n}\cup B.
\end{equation}
Firstly,
if $\omega \in A_{m}\bigcap B_{n},$
for $l>m,n$, we have \\
\begin{eqnarray*}
\int_{t-\epsilon_{l}}^{t}(D_{r}(\max_{0\leq s\leq t-\epsilon_{m}}X_{s}))^{2}dr=0, \\
\int_{t-\epsilon_{l}}^{t}(D_{r}(\max_{0\leq s \leq t-\epsilon_{n}}V_{s}))^{2}dr=0.
\end{eqnarray*}
This gives
\begin{eqnarray*}
\lim_{l\rightarrow\infty}\frac{1}{\epsilon_{l}}\int_{t-\epsilon_{l}}^{t}\alpha^{2}(D_{r}\max_{0\leq s\leq t}X_{s})^{2}dr=0,\
\lim_{l\rightarrow\infty}\frac{1}{\epsilon_{l}}\int_{t-\epsilon_{l}}^{t}(D_{r}(\max_{0\leq s\leq t}V_{s}))^{2}dr=0,
\end{eqnarray*}
a.e. on $A_{m}\cap B_{n}.$\\
Hence,
\begin{eqnarray}
\lim_{l\rightarrow \infty}\frac{1}{\epsilon_{l}}\int_{t-\epsilon_{l}}^{t}(D_{r}(X_{t}))^{2}dr
\geq \frac{1}{2}\sigma(X_{t})^{2}>0,
\end{eqnarray}
on $A_{m}\cap B_{n}.$\\
Secondly,
if $\omega \in A_{m}\bigcap B$, for fixed $m\geq 1$,
\begin{eqnarray*}
X_{t}=\int_{0}^{t}\sigma(X_{s})dB_{s}+\int_{0}^{t}b(X_{s})ds+\alpha \max_{0\leq s\leq t}X_{s}+[-(\int_{0}^{t}\sigma(X_{s})dB_{s}+\int_{0}^{t}b(X_{s})ds+\alpha\max_{0\leq s\leq t}X_{s})\vee 0].
\end{eqnarray*}
If $\int_{0}^{t}\sigma(X_{s})dB_{s}+\int_{0}^{t}b(X_{s})ds+\alpha \max \limits_{0\leq s\leq t}X_{s} >0,$ then $X_{t}=\int_{0}^{t}\sigma(X_{s})dB_{s}+\int_{0}^{t}b(X_{s})ds+\alpha \max\limits_{0\leq s\leq t}X_{s}$. In this case, we can see from the proof in Section 3 that $||DX_{t}||_{H}^{2}>0$.\\
If $\int_{0}^{t}\sigma(X_{s})dB_{s}+\int_{0}^{t}b(X_{s})ds+\alpha \max \limits_{0\leq s\leq t}X_{s}\leq 0,$ then $X_{t}=0.$\\
But $\{X_{t}=0\}$ is an event with probability zero. Indeed, according to Lemma 4.2, $0\leq Y_{t}\leq X_{t}$.
According to Proposition 4.1 in \cite{LNS}, the law of $Y_{t}$ is absolutely continuous with respect to Lesbegue measure, and then we have $P(Y_{t}=0)=0$. Therefore, $P(X_{t}=0)\leq P(Y_{t}=0)=0$.\\
Thirdly,
if $\omega \in A\bigcap B_{n}$, for fixed $n\geq1$,
\begin{eqnarray*}
D_{r}X_{t}=\sigma(X_{r})+\int_{r}^{t}D_{r}(\sigma(X_{s}))dB_{s}+\int_{r}^{t}D_{r}(b(X_{s}))ds+\alpha D_{r}(X_{t})+D_{r}(\max_{0\leq s\leq t-\epsilon_{n}}V_{s}).
\end{eqnarray*}
Hence,
\begin{eqnarray*}
(1-\alpha)D_{r}X_{t}=\sigma(X_{r})+\int_{r}^{t}D_{r}(\sigma(X_{s}))dB_{s}+\int_{r}^{t}D_{r}(b(X_{s}))ds
+D_{r}(\max_{0\leq s\leq t-\epsilon_{n}}V_{s}).
\end{eqnarray*}
Thus, for $l>n$,
\begin{eqnarray*}
\frac{1}{\epsilon_{l}}\int_{t-\epsilon_{l}}^{t}(1-\alpha)^{2}(D_{r}X_{t})^{2}dr
   &\geq& \frac{1}{2}\sigma(X_{t})^{2}-\frac{3}{\epsilon_{l}}\int_{t-\epsilon_{l}}^{t}[\int_{r}^{t}D_{r}(\sigma(X_{s}))dB_{s}]^{2}dr\\
   &&-\frac{3}{\epsilon_{l}}\int_{t-\epsilon_{l}}^{t}[\int_{r}^{t}D_{r}(b(X_{s}))ds]^{2}dr
     -\frac{3}{\epsilon_{l}}\int_{t-\epsilon_{l}}^{t}[D_{r}(\max_{0\leq s\leq t-\epsilon_{n}}V_{s})]^{2}dr.
\end{eqnarray*}
This implies,
\begin{eqnarray}
\lim_{l\rightarrow\infty}\frac{1}{\epsilon_{l}}\int_{t-\epsilon_{l}}^{t}(1-\alpha)^{2}(D_{r}X_{t})^{2}dr
\geq \frac{1}{2}\sigma(X_{t})^{2}>0 \ \  \mbox{    on a.e. } A\cap B_{n}.
\end{eqnarray}
Finally, let
$\omega \in A\bigcap B.$
Then
\begin{eqnarray}
X_{t}=\int_{0}^{t}\sigma(X_{s})dB_{s}+\int_{0}^{t}b(X_{s})ds+\alpha X_{t}+L_{t},\\
L_{t}=-(\int_{0}^{t}\sigma(X_{s})dB_{s}+\int_{0}^{t}b(X_{s})ds+\alpha X_{t})\vee 0.
\end{eqnarray}
If $\int_{0}^{t}\sigma(X_{s})dB_{s}+\int_{0}^{t}b(X_{s})ds+\alpha X_{t}\geq 0$, then
$L_{t}=0$, and
$X_{t}=\int_{0}^{t}\sigma(X_{s})dB_{s}+\int_{0}^{t}b(X_{s})ds+\alpha X_{t}.$
In this case we see that $||DX_{t}||_{H}^{2}>0$ from the proof in section 3.\\
If $\int_{0}^{t}\sigma(X_{s})dB_{s}+\int_{0}^{t}b(X_{s})ds+\alpha X_{t}<0$, then
$L_{t}=-(\int_{0}^{t}\sigma(X_{s})dB_{s}+\int_{0}^{t}b(X_{s})ds+\alpha X_{t}),$ and $X_{t}=0$.
But $X_{s}\leq X_{t}$ for $0\leq s\leq t$ on $A$. Therefore we deduce that $X_{s}=0$, for $ 0\leq s\leq t$.\\
Note that
\begin{equation}
X_{s}=\int_{0}^{s}\sigma(X_{u})dB_{u}+\int_{0}^{s}b(X_{u})du+\alpha X_{s}+L_{s}.
\end{equation}
Thus we have
\begin{equation}
-\int_{0}^{s}\sigma(X_{u})dB_{u}=\max_{0\leq u\leq s}\{-(\int_{0}^{u}\sigma(X_{v})dB_{v}+\int_{0}^{u}b(X_{v})dv+\alpha X_{u})\vee 0\}+\int_{0}^{s}b(X_{u})du,\ \ s\leq t. \label{contradictory}
\end{equation}
Notice that the right side is a process of bounded variation, so the equation (\ref{contradictory}) is not possible.
Combining all the cases, we get
$||DX_{t}||_{H}^{2}>0$. a.s.\ $\Box$
\end{proof}

\end{document}